\documentclass[reqno]{amsart}
\usepackage{geometry}
\usepackage{xcolor,graphicx} 
\usepackage{hyperref,amsfonts,amsmath,amssymb,amsthm,cleveref,tikz-cd,quiver}
\theoremstyle{definition}
\newtheorem{thm}{Theorem}[section]

\newtheorem{lem}[thm]{Lemma}
\newtheorem{prop}[thm]{Proposition}
\newtheorem{rmk}[thm]{Remark}

\numberwithin{equation}{section}
\crefname{thm}{theorem}{theorems}
\crefname{cor}{corollary}{corollaries}
\crefname{lem}{lemma}{lemmata}
\crefname{prp}{proposition}{propositions}

\newcommand{\N}{\mathbb{N}}

\newcommand{\Q}{\mathbb{Q}}

\newcommand{\Z}{\mathbb{Z}}

\newcommand{\mc}{\mathcal}
\newcommand{\mf}{\mathfrak}

\newcommand{\mbb}{\mathbb}
\newcommand{\nat}{\mbb N}
\newcommand{\cc}{\mathbb C}
\newcommand{\rr}{\mathbb R}
\newcommand{\qq}{\mbb Q}

\newcommand{\zz}{\mbb Z}
\newcommand{\cb}[1]{\left\{{#1}\right\}}

\newcommand{\rb}[1]{\left({#1}\right)}

\newcommand{\disc}{\operatorname{disc}}

\newcommand{\Tr}{\operatorname{Tr}}

\title{Pythagoras numbers for infinite algebraic fields}
\author{Nicolas Daans}
\author{Stevan Gajović}
\author{Siu Hang Man}
\author{Pavlo Yatsyna}
\address[N.~D.]{KU Leuven, Faculty of Science, Department of Mathematics, Celestijnenlaan 200B, 3001 Heverlee, Belgium}
\address[S.~G.]{MPIM Bonn, Vivatsgasse 7, 53111 Bonn, Germany}
\address[N.~D., S.~G., S.~H.~M., P.~Y.]{Charles University, Faculty of Mathematics and Physics, Department of Algebra, Sokolov\-ská~83, 186~75 Praha~8, Czech Republic}
\email[N.~Daans]{nicolas.daans@kuleuven.be}
\email[S.~Gajović]{stevangajovic@gmail.com}
\email[S.~H.~Man]{shman@karlin.mff.cuni.cz}
\email[P.~Yatsyna]{p.yatsyna@matfyz.cuni.cz}
\date{}

\begin{document}
\begin{abstract}
    We prove that the Pythagoras number of the ring of integers of the compositum of all real quadratic fields is infinite. The same holds for certain infinite totally real cyclotomic fields. In contrast, we construct infinite degree totally real algebraic fields whose rings of integers have finite Pythagoras numbers, namely, one, two, three, and at least four.
\end{abstract}
\maketitle

\section{Introduction}

Every natural number is a sum of four squares of integers, and consequently, every positive rational number is a sum of four squares of rational numbers.
This is known as Lagrange's four-square theorem.
When considering a number field $K$ (or, more generally, an algebraic field extension $K/\qq$), one may call an element $x \in K$ \emph{totally positive} if its image under every embedding $K \to \rr$ is positive.
Clearly, any sum of squares must be totally positive, and Siegel proved that, conversely, every totally positive element in such a field $K$ is a sum of four squares of elements of $K$ \cite{Siegel21}.
Infamously, the analogous question for the ring of integers $\mc{O}_K$ is significantly more subtle.

For one, for a number field $K$, in general not every totally positive element of $\mc{O}_K$ is a sum of squares of elements of $\mc{O}_K$. This has led to the study of other ways to represent totally positive elements of $\mc{O}_K$ by a positive-definite quadratic form, i.e.~the search for \emph{universal quadratic forms}.
See, e.g.~\cite{Earnest,Kala,Kim}.
The present paper is concerned with the complementary question: if we restrict our attention to those elements of $\mc{O}_K$ that can be represented as a sum of squares, \textit{how many squares} do we need?
Formally: for a commutative ring $R$ and a natural number $n$, we write ${\displaystyle \Sigma^n R^2} = \lbrace x_1^2 + \cdots + x_n^2 \mid x_1, \ldots, x_n \in R \rbrace$, and one defines the \emph{Pythagoras number of $R$} as
\[P(R) = \inf \left\lbrace n \in \nat \;\middle|\; \Sigma^n R^2 = \Sigma^{n+1} R^2 \right\rbrace \in \nat \cup \lbrace \infty \rbrace.\]
By the above discussion, we have $P(\zz) = P(\qq) = 4$ (after observing that e.g.~$7$ is not a sum of three squares in $\zz$ or $\qq$), and in general $P(K) \leq 4$ for an algebraic extension $K/\qq$ (see~\cite{Leep} for a more extensive discussion on Pythagoras numbers of fields).
When $R$ is a domain with fraction field $K$, then one easily sees that $P(R) \geq P(K)$, but the values may differ arbitrarily.
For example, $P(\qq(X)) = 5$ \cite{Pourchet}, whereas $P(\zz[X]) = \infty$ \cite{CDLR}.

Precise Pythagoras numbers of rings of integers have been computed for some number fields of small degree, see e.g.~\cite{Pet73,Krasensky,KRS,HeHu22,Tinkova}.
For any number field $K$ one has $P(\mc{O}_K) \geq 3$, and this bound is attained, e.g.~for $K = \qq(\sqrt{2})$.
For number fields $K$ (and consequently for algebraic extensions $K/\qq$) whose normal closure cannot be embedded into $\rr$, one also has an upper bound $P(\mc{O}_K) \leq 4$ (by the local-global principle for spinor genera and the fact that all valuation rings in local fields have Pythagoras number at most $4$, see \cite{Hsia-Representation-indefinite}).
Algebraic extensions $K/\qq$ whose normal closure can be embedded into $\rr$ are called \emph{totally real}.
For these fields, it is only known that $P(\mc{O}_K)$ can be bounded as a function of the degree $[K : \qq]$ when the latter is finite, see \cite{KY}; the argument relies on studying sums of squares of integral linear forms and bounds on the integral Mordell's function, see \cite{KO,BCIL}.
We do not know, for totally real number fields $K$, whether $P(\mc{O}_K)$ \textit{must} grow with the degree $[K : \qq]$, in fact, we know almost nothing about the expected behaviour of $P(\mc{O}_K)$ as $[K : \qq]$ grows, see \Cref{sect:Open} for several open problems.

In this paper, we take a first step towards understanding this question by considering the limit cases, i.e.~looking at totally real fields of infinite degree. For simplicity, when we speak of an \emph{algebraic field} we mean a field $K$ which is an algebraic extension of $\qq$, and when we speak of an \emph{infinite algebraic field} we mean an algebraic field of infinite degree over $\qq$. In \cite{Scharlau}, an infinite chain of number fields $(K^i)_{i \in \nat}$ was constructed with unbounded $P(\mc{O}_{K^i})$.
Inspection of the proof reveals that some well-chosen subset $S \subseteq \zz$ is constructed for which the field $K = \qq(\sqrt{n} \mid n \in S )$ satisfies $P(\mc{O}_K) = \infty$.
We show that the infinitude of the Pythagoras number holds for many choices of $S$, and conjecture that it holds for any infinite set $S$ which produces an infinite algebraic field $\qq(\sqrt{n} \mid n \in S)$:

\begin{thm}\label{thm:multiquad}
Let $S \subseteq \nat$ contain infinitely many prime numbers $p \equiv 1 \pmod 4$, set $K = \qq( \sqrt{n} \mid n \in S)$.
Then
    $P(\mc O_{K})=\infty$.
\end{thm}
The above theorem applies in particular to the field $\mbb Q^{(2), +} = \qq(\sqrt{n} \mid n \in \nat)$, the compositum of all real quadratic fields.
Fields of the form $\qq(\sqrt{n} \mid n \in S)$ satisfy the Northcott property, see, e.g.~\cite{BZ}.
In the recent work \cite{DKM} it was shown that totally real fields of infinite degree and with the Northcott property do not admit universal quadratic forms, and one may speculate on a further relation with the Pythagoras number of the ring of integers.

We provide another class of totally real algebraic fields for which the ring of integers has an infinite Pythagoras number but without the Northcott property.
For $N\in\N$, let $\zeta_N = e^{2\pi i/N}$. Let $K_N= \Q(\zeta_N)$ be the \textit{$N$-th cyclotomic field}.  Pythagoras numbers of $P(K_N)$ are known in relation to the \textit{Stufe} (see, e.g.~\cite{Moser}). In light of our setup, denote by $K_N^+ = \Q(\zeta_N + \zeta_N^{-1})$ the \textit{maximal real subfield} of $K_N$. For a prime number $p$, let $K_{p^\infty} = \bigcup_{n\in\N} K_{p^n}$, and $K_{p^\infty}^+ = \bigcup_{n\in\N} K_{p^n}^+$ the maximal real subfield of $K_{p^\infty}$.
We prove the following.

 \begin{thm}\label{thm:2}
 Let $p$ be a prime number that satisfies $p\equiv 1\pmod{4}$. Then $P(\mc O_{K_{p^{\infty}}^+})=\infty$.
 \end{thm}

For completeness, in the next section of this paper we include some constructions of totally real fields $K$ of infinite degree over $\qq$ and for which $P(\mc{O}_K)$ is finite.
Except for the case $P(K) = P(\mc{O}_K) = 1$ (i.e.~the field $K$ is \emph{pythagorean}), the constructions are surprisingly subtle.

\begin{thm}[see \Cref{prop:finite-PN}]
    One, two, and three appear as Pythagoras numbers of the rings of integers of infinite totally real algebraic fields.
\end{thm}

\section{Fields with small Pythagoras numbers} 

The construction idea is each time as follows: for a field $K$ of characteristic $0$ one can consider its \emph{maximal totally real subfield} $K^+$, by which we mean the field of all elements of $K$ which are algebraic over $\qq$ and are totally real. We will then consider $\mc{O}_{K^+}$, the ring of integers of $K^+$.  

\begin{prop}\label{prop:finite-PN}
In each of the following cases, we consider a field $K$, its maximal totally real subfield $K^+$, and the corresponding ring of integers $\mc{O}_{K^+}$.
\begin{enumerate}
\item\label{it:C} If $K = \cc$, then $P(\mc{O}_{K^+}) = 1$,
\item\label{it:Q3unr} if $p$ is an odd prime and $K$ is the maximal unramified extension of $\qq_p$, then $P(\mc{O}_{K^+}) = 2$,
\item\label{it:Q3} if $K = \qq_3$, then $P(\mc{O}_{K^+}) = 3$.
\item\label{it:Q2} if $K = \qq_2$, then $4\leq P(\mc{O}_{K^+}) \leq 6$.
\end{enumerate}
Furthermore, in cases \eqref{it:C} and \eqref{it:Q3unr}, every totally positive element of $\mc{O}_{K^+}$ is a sum of squares, whereas in cases \eqref{it:Q3} and \eqref{it:Q2}, there exist totally positive elements which are not a sum of squares.
\end{prop}
\begin{proof}

\eqref{it:C}:
Let $\alpha \in \mc{O}_{K^+}$ be totally positive.
Then $\sqrt{\alpha}$ is totally real and contained in $\cc$, hence also contained in $\mc{O}_{K^+}$.
This shows that every totally positive element in $\mc{O}_{K^+}$ is a square.

\eqref{it:Q3unr}:
Let $\alpha \in \mc{O}_{K^+}$ be totally positive.
We may write $\alpha = p^k u$ for certain $k \in \Z_{\ge0}$ and $u \in K$ with $v_p(u) = 0$, where $v_p$ is the $p$-adic valuation on $K$.
Since $p \in \qq \subseteq K^+$ and $p$ is totally positive, we obtain that $u \in K^+$ and is also totally positive. Moreover, since $\alpha \in \mc{O}_{K^+}$, we have $v_{\mathfrak p}(u) = v_{\mathfrak p}(p^{-k}\alpha)\ge 0$ for every prime $\mf p$ of $\Q(\alpha)$ (where $v_{\mf p}$ denotes the $\mf p$-adic valuation on the number field $\qq(\alpha)$), hence $u \in \mc{O}_{K^+}$.
Since $v_p(u) = 0$, by Hensel's Lemma there exists $v \in K$ with $u = v^2$; since $u$ is totally positive we must have that $v \in K^+$, and as $u \in \mc{O}_{K^+}$, also $v \in \mc{O}_{K^+}$.
We conclude that, in $\mc{O}_{K^+}$, $\alpha$ is a power of $p$ times a square, and since $p$ is a sum of two squares (e.g.~because $\sqrt{p-1} \in \mc{O}_{K^+}$), we conclude that every totally positive element of $\mc{O}_{K^+}$ is a sum of two squares.

Since on the other hand $p$ is not a square in $K$ and hence also not in $\mc{O}_{K^+}$, we obtain $P(\mc{O}_{K^+}) = 2$.

\eqref{it:Q3}:
Note that $3$ is a sum of $3$ squares in $\zz$ but not a sum of $2$ squares in $\qq_3$, hence in particular not in $\mc{O}_{K^+}$.
Furthermore, the totally positive unit $8+3\sqrt{7} \in \mc{O}_{K^+}$ is not a square in $\qq_3$, hence also not in $\mc{O}_{K^+}$, but then it cannot be a sum of squares in $\mc{O}_{K^+}$ since by H{\"o}lder's inequality (applied to the norm) it is clear that a unit cannot be expressed as a sum of two or more 
totally positive integers (see, e.g. \cite[3.1]{OMeara}).

It remains to show that a sum of $4$ squares in $\mc{O}_{K^+}$ is a sum of $3$ squares.
So, consider $\alpha_1, \alpha_2, \alpha_3, \alpha_4 \in \mc{O}_{K^+}$ and $\beta = \alpha_1^2 + \alpha_2^2 + \alpha_3^2 + \alpha_4^2$.
Write $\alpha_i = 3^{k_i} u_i$ for some $k_i \in \Z_{\ge0}$ and $u_i \in \mc{O}_{K^+}$ not divisible by $3$. By dividing all the $\alpha_i$'s by a common power of $3$ and switching their order, we may assume $0 = k_1 \leq k_2 \leq k_3 \leq k_4$.
We consider two cases, and show that in each case, $\beta$ is a sum of $3$ squares.
\begin{itemize}
\item $k_4 \geq 1$: In this case, $\alpha_1^2 + \alpha_4^2 \equiv 1 \pmod 3$, and by Hensel's Lemma $\alpha_1^2 + \alpha_4^2$ is a square in $K$.
Since it is additionally a totally positive element of $\mc{O}_{K^+}$, it is in fact a square in $\mc{O}_{K^+}$, making $\beta$ a sum of $3$ squares,
\item $k_4 = 0$: in this case, also $k_2 = k_3 = 0$.
We have that $\alpha_i^2 \equiv 1 \pmod 3$ for each $i$ and thus $\beta \equiv 1 \pmod 3$; as above, we conclude by Hensel's Lemma that $\beta$ is itself a square.
\end{itemize}
\eqref{it:Q2}: Since 7 is a sum of four squares, but not a sum of three squares in $\Q_2$, we have $P(\mc{O}_{K^+}) \geq 4$.
As in \eqref{it:Q3} there is a totally positive unit $23 + 4\sqrt{33} \in \mc{O}_{K^+}$ which is not a square (since it is not a square modulo $4$), hence not a sum of squares either.
It remains to prove that the sum of any seven squares is a sum of at most six squares.

As in \eqref{it:Q3}, consider $\beta=\alpha_1^2+\cdots+\alpha_7^2$, and we may assume that $\alpha_i=2^{k_i}u_i$, for $1\leq i\leq 7$ such that $0=k_1\leq k_2\leq\cdots\leq k_7$ and $v_2(u_i)=0$. If $k_7\geq 2$, then $\alpha_1^2+\alpha_7^2\equiv 1\pmod{8}$ is a totally positive number, so it is a square in $\mc{O}_{K^+}$. If $k_6=k_7=1$, then $\alpha_1^2+\alpha_6^2+\alpha_7^2$ is a square in $\mc{O}_{K^+}$. If $k_7=1$, but $k_6=0$, then $\alpha_2^2+\cdots+\alpha_6^2+\alpha_7^2$ is a square in $\mc{O}_{K^+}$. Finally, assume $k_7=0$. Then for each $1\leq i\leq 7$, we have $\alpha_i^2=8\gamma_i+1$, for some $\gamma_i\in \mc{O}_{K}$. We apply Lemma \ref{lem:divisible-by-4} below to conclude that there are four $\gamma_i$ such that 4 divides their sum, say $\gamma_1,\ldots,\gamma_4$. Then $\alpha_1^2+\cdots+\alpha_4^2\equiv 4\pmod{32}$, hence it is a square in $\Z_2$, and since it is totally positive, it is a square in $\mc{O}_{K^+}$, which finishes the proof.
\end{proof}

\begin{lem}\label{lem:divisible-by-4}
Let $m_1,\ldots,m_7\in \Z/4\Z$. There are $1\leq i<j<k<l\leq 7$ such that $m_i+m_j+m_k+m_l=0$.    
\end{lem}

\begin{proof}
If four $m_i$s share the same value, then we can take them. If all values from $\Z/4\Z$ are represented, then we can take the remaining three numbers, and there will be a choice of the fourth one, which in sum gives zero.

Thus, we may assume that there are at most three values among the numbers  $m_1,\ldots,m_7\in \Z/4\Z$. By the pigeonhole principle, there are at least three same values, and without loss of generality, we may assume $m_1=m_2=m_3=0$ and that the others are nonzero. If there are two values 2, we are done. If there are both of values 1 and 3, we are also done. Hence, the only remaining case is when, without loss of generality, $m_4=2$, and $m_5=m_6=m_7\in\{1,3\}$. Then, $m_1+m_4+m_5+m_6=0$.
\end{proof}

\begin{rmk}
The number 7 in Lemma \ref{lem:divisible-by-4} is the smallest possible because we can consider $0,0,0,1,1,1$, so the statement is not true for six numbers. Hence, the bound 6 in Proposition \ref{prop:finite-PN}\eqref{it:Q2} is the best possible that can be proved using this strategy.
\end{rmk}

\section{Fields with infinite Pythagoras numbers}

The obvious strategy to prove that an algebraic field $K$ has infinite Pythagoras number $P(\mc O_K)$ is to find for arbitrary $n\in\N$ a sum of $n$ squares in $\mc O_K$ that cannot be written as a sum of $n-1$ squares. This constructive approach has its inspirations in \cite{Scharlau,Pollack}.

\subsection{Infinite multiquadratic fields}

Here we prove \Cref{thm:multiquad}. To do this, it is helpful to define some auxiliary notions. Throughout this subsection, let $K$ denote a subfield of $\Q^{(2),+}$. Then any $\alpha\in K$ can be written as a sum
\begin{equation}\label{eq:Q_expansion}
\alpha = \sum_{i=1}^m c_i \sqrt{d_i}, 
\end{equation}
where $m\in\Z_{\ge0}$, $c_i \in \Q\backslash\{0\}$, and $d_i\in\N$ are distinct and squarefree. Since the elements of the set $\{\sqrt{n} \mid n\in\N \text{ squarefree}\}$ are linearly independent over $\Q$, such a sum is unique up to relabelling of the indices. For $n\in\N$ squarefree, we say $\alpha$ \emph{contains} $\sqrt{n}$ if $d_i=n$ for some $i$ in the sum \eqref{eq:Q_expansion}.

\begin{lem}\label{lem:basis-sqrt-p}
Let $p\equiv 1\pmod 4$ be a prime, $L \subseteq K$ a number field containing $\sqrt{p}$.
\begin{enumerate}
\item There is a subfield $M\subseteq L$ in which $p$ is unramified, and such that $\lbrace 1, \frac{1 + \sqrt{p}}{2} \rbrace$ is an $\mc O_M$-basis of $\mc O_L$.
\item For $\beta, \gamma \in \mc{O}_M$ and $\alpha = \beta + \frac{1+\sqrt{p}}{2} \gamma$ with $\gamma \neq 0$, we have
\[
\Tr_{L/\Q}(\alpha^2)\geq \dfrac{p}{4}\Tr_{L/\Q}(\gamma^2) \geq \frac{p}{4} [L:\Q].
\]
In particular, if furthermore $\gamma\neq \pm1$, then we have
\[
\Tr_{L/\Q}(\alpha^2)\geq \dfrac{3p}{8}[L:\Q].
\]
\end{enumerate}
\end{lem}
\begin{proof}
We first prove (1). 
We may write $L=\Q(\sqrt{p},\sqrt{d_2},\ldots,\sqrt{d_m})$ for some $d_i\in\N$ squarefree. Possibly replacing $d_i$ by $d_i/p$, we may assume that $p\nmid d_2\cdots d_m$. Then, $L$ is the compositum of $\Q(\sqrt{p})$ and $M=\Q(\sqrt{d_2},\ldots,\sqrt{d_m})$ and $p$ does not divide $\disc(M)$, hence it is unramified at $M$. Since the discriminants of $\Q(\sqrt{p})$ and $M$ are coprime, one can form an integral basis for $\mc O_L$ as the product of integral bases for $\mc O_M$ and $\mc O_{\Q(\sqrt{p})}$. In particular, the latter can be taken to be $\{1,\frac{1+\sqrt{p}}{2}\}$. So, there are $\beta,\gamma\in \mc O_M$ such that $\alpha=\beta+\gamma \frac{1+\sqrt{p}}{2}$.

Now we prove~(2). We expand
\[
\alpha^2=\beta^2+\beta\gamma+\beta\gamma\sqrt{p}+\gamma^2\dfrac{1+p+2\sqrt{p}}{4}.
\]
Since $M\cap \Q(\sqrt{p})=\Q$, we have $\Tr_{L/\Q}(\beta\gamma\sqrt{p})=\Tr_{L/\Q}(\gamma^2\sqrt{p})=0$. Hence
\[
\Tr_{L/\Q}(\alpha^2)=\Tr_{L/\Q}\left(\beta^2+\beta\gamma+\dfrac{\gamma^2}{4}+\dfrac{p}{4}\gamma^2\right)=\Tr_{L/\Q}\left(\left(\beta+\dfrac{\gamma}{2}\right)^2\right)+\dfrac{p}{4}\Tr_{L/\Q}(\gamma^2)\geq \dfrac{p}{4}\Tr_{L/\Q}(\gamma^2).
\]
The final statement follows from Siegel's theorem~\cite[Theorem~III]{Siegel45} that for a totally positive integer $\delta\neq 1$, we have $\Tr_{\Q(\delta)/\Q}(\delta)\geq \frac{3}{2}[\Q(\delta):\Q]$. For $\delta=1$, it is clear that $\Tr_{L'/\Q}(1)=[L':\Q]$ for any number field $L'$.
\end{proof}

Now we are ready to establish the main ingredient of the proof of \Cref{thm:multiquad}.

\begin{lem}\label{lem:unbounded-number-of-squares}
Let $\{p_n\}_{n\in \N}$ be a sequence of prime numbers, such that $p_n\equiv 1\pmod{4}$ for all $n\in\N$, and 
\begin{equation}\label{eq:estimate-on-p_n}
p_{n+1}>2+2\sum_{i=1}^n(1+p_i) \quad \forall n\in\N.
\end{equation}
For $n\in\N$, we define
\begin{equation}\label{eq:unrepresentable-number}
s_n=\sum_{i=1}^{n}\left(\dfrac{1+\sqrt{p_i}}{2}\right)^2. 
\end{equation}
Then $s_n$ cannot be expressed as a sum of $n-1$ squares in $\mc{O}_{\Q^{(2),+}}$. 
\end{lem}
\begin{proof}
Let
\begin{equation}\label{eq:represent-smaller}
s_n=\alpha_1^2+\cdots+\alpha_m^2    
\end{equation}
be a representation of $s_n$ as a sum of $m$ squares in $\mc{O}_{\Q^{(2),+}}$. We observe that~\eqref{eq:represent-smaller} is defined over a number field $L$,
which we may assume to contain $\sqrt{p_1},\ldots,\sqrt{p_n}$.
We apply \Cref{lem:basis-sqrt-p}(1) to find a subfield $M$ of $L$ in which $p_n$ is unramified and such that $\lbrace 1, \frac{1+\sqrt{p_n}}{2}\rbrace$ is a $\mc O_M$-basis of $L$.
Since $s_n$ contains $\sqrt{p_n}$,
there is some $\alpha_i\in\mc O_L$, such that $\alpha_i^2$ contains $\sqrt{p_n}$.
Without loss of generality, we may take $i=1$, and denote $\alpha_1$ simply by $\alpha$.

It follows that we may write $\alpha = \beta + \frac{1+\sqrt{p_n}}{2}\gamma$ for some $\beta, \gamma \in \mc O_M$ with $\gamma \neq 0$.
If $\gamma\neq \pm1$, then by \Cref{lem:basis-sqrt-p}(2), we have 
\[\Tr_{L/\Q}(\alpha^2)\geq \dfrac{3p_n}{8}[L:\Q].\]
Taking $\Tr_{L/\Q}$ to \eqref{eq:represent-smaller} and keeping in mind \eqref{eq:unrepresentable-number}, we have
\[
[L:\Q]\sum_{i=1}^{n}\dfrac{1+p_i}{4}=\Tr_{L/\Q}(\alpha^2)+\sum_{j=2}^{m}\Tr_{L/\Q}(\alpha_j^2)\geq  \dfrac{3p_n}{8}[L:\Q].
\]
However, this contradicts the assumption \eqref{eq:estimate-on-p_n}.
Thus, we have that $\gamma=\pm1$, and we may suppose that $\gamma=1$. 
By a similar argument, we see that $\alpha_2^2, \ldots, \alpha_m^2$ do not contain $\sqrt{p_n}$: if, for example, $\alpha_2^2$ were to contain $\sqrt{p_n}$, then
 we could write $\alpha_2 = \beta_2 + \frac{1+\sqrt{p_n}}{2}\gamma_2$ for some $\beta_2, \gamma_2 \in \mc O_M$ with $\gamma_2 \neq 0$, and
by Lemma~\ref{lem:basis-sqrt-p}(2) we would have
\[
[L:\Q]\sum_{i=1}^{n}\dfrac{1+p_i}{4} \geq \Tr_{L/\Q}(\alpha^2)+\Tr_{L/\Q}(\alpha_2^2)\geq \dfrac{p_n}{2}[L: \Q],
\]
again contradicting assumption \eqref{eq:estimate-on-p_n}. We conclude that $\alpha_2^2, \ldots, \alpha_m^2$ do not contain $\sqrt{p_n}$, and in fact lie in $M$.

Now \eqref{eq:unrepresentable-number} and~\eqref{eq:represent-smaller} become
\begin{equation*}
s_{n-1} = \sum_{i=1}^{n-1}\left(\dfrac{1+\sqrt{p_i}}{2}\right)^2 =\beta(1+\sqrt{p_n}) +\beta^2+\sum_{j=2}^m \alpha_j^2, 
\end{equation*}
and using that $\{ 1, \sqrt{p_n} \}$ is an $M$-basis of $L$, we deduce that $\beta = 0$.
We see that $s_{n-1}$ is a sum of $m-1$ squares.
By induction on $n$, we conclude that $m \geq n$.
\end{proof}

\begin{proof}[Proof of \Cref{thm:multiquad}]
Since $S \subseteq \nat$ contain infinitely many prime numbers $p \equiv 1 \pmod 4$, we can find a sequence $\{p_n\}_{n\in\N}$ that satisfies the assumptions of \Cref{lem:unbounded-number-of-squares}, with $\sqrt{p_n}\in K= \qq( \sqrt{n} \mid n \in S)$ for each $n\in\N$. Then for each $n\in\N$ the constructed element $s_n \in K$ in \Cref{lem:unbounded-number-of-squares} is a sum of $n$ squares in $\mc O_K$, but not a sum of $n-1$ squares in $\mc O_K \subseteq \mc O_{\Q^{(2),+}}$. Thus, $P(\mc O_K)=\infty$.
\end{proof}

\begin{rmk}\label{rmk:more-general}
Let $K=\qq(\sqrt{n} \mid n \in S )$, where $S$ contains infinitely many $n\equiv 1\pmod{4}$ such that for all $m\in S$, we have $n\mid m$ or $\gcd(m,n)=1$. Then a similar proof shows that $P(\mc O_K)=\infty$.
\end{rmk}

\begin{rmk}\label{rmk:Q2-argument-doesn't-work-in-general}

In contrast to \Cref{rmk:more-general}, our method is not sufficient to prove that $P(\mc{O}_K) = \infty$ for a general infinite multiquadratic field $K$. Our strategy uses that if $p\equiv 1\pmod{4}$ 
and $\alpha \in \qq^{(2),+}$ such that $\alpha^2$ contains $\sqrt{p}$, then by Lemma~\ref{lem:basis-sqrt-p}(1), we can write $\alpha=\beta+\gamma\dfrac{1+\sqrt{p}}{2}$, where $\beta,\gamma\in \qq^{(2),+}$ are such that $\qq(\beta,\gamma)\cap\qq(\sqrt{p})=\qq$. This is the key step, which we use later to prove that if $\gamma\neq \pm1$, then by Lemma~\ref{lem:basis-sqrt-p}(2)
\begin{equation}\label{eq:trace-inequality}
\frac{1}{[\Q(\alpha) : \Q]}\left(\Tr_{\qq(\alpha)/\qq}(\alpha^2)-\Tr_{\qq(\alpha)/\qq}\left(\dfrac{1+\sqrt{p}}{2}\right)^2\right)\geq \frac{p-2}{8}.    
\end{equation}
We use this inequality to choose large enough $p$ to prove that we must have $\gamma=\pm1$ in this setting. However, the analogue of Lemma~\ref{lem:basis-sqrt-p}(1) is not true for composite numbers, hence the condition in Remark~\ref{rmk:more-general}. 

We can find examples where inequality \eqref{eq:trace-inequality} is not satisfied, even stronger, that the trace of $\alpha^2$ can be small.
Let $p_1<p_2<p_3$ be three primes congruent to 1 modulo 4 and with $p_2 < 2p_1$.

Let $K=\Q(\sqrt{p_1p_2},\sqrt{p_1p_3})$ and let $$\alpha=\frac{\sqrt{p_1p_2}+\sqrt{p_1p_3}+\sqrt{p_2p_3}+1}{4}\in \mc O_K.$$
Then $K=\Q(\alpha)$ and $\alpha^2$ contains $\sqrt{p_1p_3}$, but our assumptions on $p_1,p_2,p_3$ imply
$$\Tr_{\qq(\alpha)/\qq}(\alpha^2)=\dfrac{p_1p_2+p_1p_3+p_2p_3+1}{4}<1+p_1p_3=\Tr_{\qq(\alpha)/\qq}\left(\dfrac{1+\sqrt{p_1p_3}}{2}\right)^2.$$
Consequently, we cannot use our proof technique, for example,
to show that $P(\mc{O}_K) = \infty$ when $K = \qq(\sqrt{pq} \mid p \neq q \text{ prime number})$.
\end{rmk}

\subsection{Infinite real cyclotomic fields}

We prove \Cref{thm:2} by a careful study of the arithmetic of real cyclotomic fields $K_{p^n}^+$. We establish the following proposition, which immediately implies \Cref{thm:2}.
\begin{prop}\label{prp:cyclotomic_many_squares}
Let $p\equiv 1\pmod{4}$ be a prime. For $2\le m \in\N$, we define
\[
t_m = \sum_{k=2}^m \rb{\zeta_{p^k}+\zeta_{p^k}^{-1}}^2.
\]
Then $t_m$ cannot be expressed as a sum of $m-2$ squares in $\mc O_{K_{p^\infty}^+}$.
\end{prop}
The rest of the section is devoted to the proof of \Cref{prp:cyclotomic_many_squares}. Let $p\equiv 1\pmod{4}$ be a prime, $n\in\N$, and write $\omega(k) = \omega(p^n,k) := \zeta_{p^n}^k + \zeta_{p^n}^{-k}$, for $k\in\Z$. It is well known that
\[
[K_{p^n}^+:\Q] = \frac{\varphi(p^n)}2 = \frac{(p-1)p^{n-1}}2,
\]
and $\mc O_{K_{p^n}^+} = \Z[\omega(1)]$ (see, e.g. Proposition~2.16 in \cite{Washington}). It follows that the set
\begin{equation}\label{eq:cyclotomic_basis}
\cb{1,\omega(1),\omega(2),\ldots,\omega(\tfrac{p-1}2 p^{n-1}-1)}
\end{equation}
is a $\Z$-basis of $\mc O_{K_{p^n}^+}$. We shall also make use of the following formulas, which are easy to verify:
\begin{align}
\omega(k)\omega(l) &= \omega(k+l) + \omega(k-l),\nonumber\\
\omega(p^n-k) &= \omega(k), \label{eq:rf}\\
\omega(\tfrac{p-1}2 p^{n-1}) &= - 1 - \sum_{l=1}^{\frac{p-3}2} \omega(lp^{n-1}),\label{eq:cyclotomic_decay}\\
\omega(\tfrac{p-1}2 p^{n-1}+s) &= - \sum_{l=0}^{\frac{p-3}2} \omega(lp^{n-1}+s) - \sum_{l=1}^{\frac{p-1}2} \omega(lp^{n-1}-s), & &(1\le s \le \tfrac{p^{n-1}-1}2). \label{eq:d2} 
\end{align}

Let $\alpha \in \mc O_{K_{p^n}^+}$. In terms of the basis \eqref{eq:cyclotomic_basis}, we may write $\alpha$ as a linear combination
\begin{equation}\label{eq:cyclotomic_expansion}
\alpha = a_0 + \sum_{k=1}^{\frac{p-1}2 p^{n-1}-1} a_k \omega(k), \quad a_k \in \Z.
\end{equation}
For $0\le k \le \frac{p-1}2 p^{n-1}-1$, we write $C_k(\alpha):= a_k$ the $k$-th coefficient of $\alpha$ in the expansion \eqref{eq:cyclotomic_expansion}. For $\alpha$ given as in \eqref{eq:cyclotomic_expansion}, its square is given by
\begin{equation}\label{eq:a2_expand}
\alpha^2 = a_0^2 + \sum_{k=1}^{\frac{p-1}2 p^{n-1}-1} a_k^2(\omega(2k)+2) + 2 \sum_{k=1}^{\frac{p-1}2 p^{n-1}-1} a_0a_k \omega(k) + 2\sum_{k> l \ge 1}^{\frac{p-1}2 p^{n-1}-1} a_ka_l(\omega(k+l)+\omega(k-l)).
\end{equation}
In particular, we have
\begin{equation}\label{eq:cyclotomic_mod2}
\alpha^2 \equiv a_0^2 + \sum_{k=1}^{\frac{p-1}2 p^{n-1}-1} a_k^2\omega(2k) \pmod{2}.
\end{equation}

We compute the constant coefficient $C_0(\alpha^2)$ of $\alpha^2$. Using \eqref{eq:rf}, \eqref{eq:cyclotomic_decay}, and \eqref{eq:d2}, we find that for $1\le k \le p^n-1$, we have
\[
C_0(\omega(k)) = \begin{cases} -1 & \text{if } k=\frac{p\pm 1}2 p^{n-1},\\ 0 & \text{otherwise.}\end{cases}
\]
Plugging this into \eqref{eq:a2_expand}, noting that $p\equiv 1\pmod{4}$, yields
\begin{equation}\label{eq:ctb_pre}
C_0(\alpha^2) = a_0^2 + \sum_{k=1}^{\frac{p-1}2 p^{n-1}-1} 2a_k^2 - a_{\frac{p-1}4p^{n-1}}^2 - 2\sum_{\substack{k>l\ge 1\\ k+l = \frac{p\pm 1}2 p^{n-1}}}^{\frac{p-1}2 p^{n-1}-1} a_ka_l.
\end{equation}
Observing that 
\begin{align*} \sum_{\substack{k>l\ge 1\\ k+l = \frac{p-1}2 p^{n-1}}}^{\frac{p-1}2 p^{n-1}-1} (a_k-a_l)^2 &= \sum_{\substack{k=1\\ k\ne \frac{p-1}4 p^{n-1}}}^{\frac{p-1}2p^{n-1}-1} a_k^2 - 2\sum_{\substack{k>l\ge 1\\ k+l = \frac{p-1}2 p^{n-1}}}^{\frac{p-1}2 p^{n-1}-1} a_ka_l,\\
\intertext{and}
\sum_{\substack{k>l\ge 1\\ k+l = \frac{p+1}2 p^{n-1}}}^{\frac{p-1}2 p^{n-1}-1} (a_k-a_l)^2 &= \sum_{\substack{k=p^{n-1}+1}}^{\frac{p-1}2p^{n-1}-1} a_k^2 - 2\sum_{\substack{k>l\ge 1\\ k+l = \frac{p+1}2 p^{n-1}}}^{\frac{p-1}2 p^{n-1}-1} a_ka_l,
\end{align*}
it follows that \eqref{eq:ctb_pre} can be rewritten as
\[
C_0(\alpha^2) = \sum_{k=0}^{p^{n-1}} a_k^2 + \sum_{\substack{k>l\ge 1\\ k+l = \frac{p\pm 1}2 p^{n-1}}}^{\frac{p-1}2 p^{n-1}-1} (a_k-a_l)^2.
\]
For $0\le w \le \frac{p^{n-1}-1}2$, define
\begin{align*}
A_w(\alpha) := &a_w^2 + a_{p^{n-1}-w}^2 + \sum_{\substack{k>l\ge 1\\ k+l = \frac{p\pm 1}2 p^{n-1}\\ k\equiv \pm w\pmod{p^{n-1}}}}^{\frac{p-1}2 p^{n-1}-1} (a_k-a_l)^2,
\end{align*}
so we have
\[
C_0(\alpha^2) = \sum_{k=0}^{\frac{p^{n-1}-1}2} A_k(\alpha).
\]
For $1\le w \le \frac{p^{n-1}-1}2$, we can rearrange the sum $A_w(\alpha)$ by placing terms with the same coefficients next to each other, and write
\begin{equation}\label{eq:Aw}\begin{aligned}
A_w(\alpha) = &a_w^2 + (a_w - a_{\frac{p-1}2 p^{n-1}-w})^2 + (a_{\frac{p-1}2 p^{n-1}-w} - a_{p^{n-1}+w})^2 + (a_{p^{n-1}+w} - a_{\frac{p-3}2 p^{n-1}-w})^2\\
&\quad + \cdots + (a_{2p^{n-1}-w} - a_{\frac{p-3}2 p^{n-1}+w})^2 + (a_{\frac{p-3}2 p^{n-1}+w} - a_{p^{n-1}-w})^2 + a_{p^{n-1}-w}^2, 
\end{aligned}\end{equation}
We will find it helpful to consider the expansion in \eqref{eq:Aw} in a more combinatorial fashion.
The set $I_w$ of integers appearing as an index in the expansion \eqref{eq:Aw} consists precisely of those integers $k$ with $1 \leq k \leq \frac{p-1}{2}p^{n-1}-1$ for which $k \equiv \pm w \pmod{p^{n-1}}$.
We have $\lvert I_w \rvert = p-1$. We fix a dummy element $\ast$, and define on $I_w^\ast = I_w \cup \lbrace \ast \rbrace$ the cyclic permutation
\begin{equation}\label{eq:R-defi}
R_w : I_w^\ast \to I_w^\ast , \quad x \mapsto \begin{cases}
\ast &\text{if } x = p^{n-1} - w, \\
w &\text{if } x = \ast, \\
\frac{p-1}{2}p^{n-1} - x &\text{if } x \equiv w \pmod {p^{n-1}}, \\
\frac{p+1}{2}p^{n-1} - x &\text{if } x \equiv -w \pmod {p^{n-1}} \text{ and } x \neq p^{n-1} - w.
\end{cases}
\end{equation}
If we take the convention that $a_{\ast} = 0$, the equality \eqref{eq:Aw} can then be rewritten as
\begin{equation}\label{eq:Aw-2}
A_w(\alpha) = \sum_{k \in I_w^\ast} (a_k - a_{R_w(k)})^2 = \sum_{i=0}^{p-1} (a_{R^i_w(k_0)} - a_{R_w^{i+1}(k_0)})^2
\end{equation}
where $k_0 \in I_w^\ast$ is any fixed element.
We can make this more visual as follows.
We can define a graph with set of vertices $I_w$, and where two vertices are adjacent if one gets mapped to the other by $R_w$ (or, equivalently, if they appear in the same bracketed term in \eqref{eq:Aw}). This gives the following:
$$
\begin{tikzcd}[column sep=small]
w
  \arrow[r, bend left=30, "R_w"]
  & {\tfrac{p-1}{2}p^{n-1} - w}
    \arrow[l, bend left=30, "R_w^{-1}"]
  \arrow[r, bend left=30, "R_w"]
  & {p^{n-1} + w}
    \arrow[l, bend left=30, "R_w^{-1}"]
  \arrow[r, bend left=30, "R_w"]
  & {\tfrac{p-3}{2}p^{n-1} - w}
    \arrow[l, bend left=30, "R_w^{-1}"]
  \arrow[r, bend left=30, "R_w"]
  & \cdots
    \arrow[l, bend left=30, "R_w^{-1}"]
  \arrow[r, bend left=30, "R_w"]
  & {\tfrac{p-3}{2}p^{n-1} + w}
    \arrow[l, bend left=30, "R_w^{-1}"]
  \arrow[r, bend left=30, "R_w"]
  & {p^{n-1} - w}
    \arrow[l, bend left=30, "R_w^{-1}"],
\end{tikzcd}
$$
Given that $R_w^{-2}(w)=R_w^{-1}(\ast)=p^{n-1}-w$ and $R_w^2(p^{n-1}-w)=R_w(\ast)=w$, the graph can naturally be extended to a cycle of length $p$:

\begin{equation}\label{eq:diagram}
\begin{tikzcd}
	& {p^{n-1} + w} & \cdots & {2p^{n-1} - w} \\
	{\frac{p-1}{2}p^{n-1} - w} &&&& {\frac{p-3}{2}p^{n-1}+w} \\
	& w & \ast & {p^{n-1}-w}
	\arrow["{R_w}", curve={height=-12pt}, from=1-2, to=1-3]
	\arrow["{R^{-1}_w}", curve={height=-2.5pt}, from=1-2, to=2-1]
	\arrow["{R^{-1}_w}", curve={height=-12pt}, from=1-3, to=1-2]
	\arrow["{R_w}", curve={height=-12pt}, from=1-3, to=1-4]
	\arrow["{R^{-1}_w}", curve={height=-12pt}, from=1-4, to=1-3]
	\arrow["{R_w}", curve={height=-12pt}, from=1-4, to=2-5]
	\arrow["{R_w}", curve={height=-12pt}, from=2-1, to=1-2]
	\arrow["{R^{-1}_w}", curve={height=-6pt}, from=2-1, to=3-2]
	\arrow["{R^{-1}_w}", curve={height=-6pt}, from=2-5, to=1-4]
	\arrow["{R_w}", curve={height=-18pt}, from=2-5, to=3-4]
	\arrow["{R_w}", curve={height=-18pt}, from=3-2, to=2-1]
	\arrow["{R^{-1}_w}", curve={height=-12pt}, from=3-2, to=3-3]
	\arrow["{R_w}", curve={height=-12pt}, from=3-3, to=3-2]
	\arrow["{R^{-1}_w}", curve={height=-12pt}, from=3-3, to=3-4]
	\arrow["{R^{-1}_w}", curve={height=-2.5pt}, from=3-4, to=2-5]
	\arrow["{R_w}", curve={height=-12pt}, from=3-4, to=3-3]
\end{tikzcd}\
\end{equation}
Clearly, for any vertex $k$, $R_w^i(k)$ visits every possible vertex in the diagram exactly once, for $i=0,\ldots,p-1$, where we use the notation $R_w^0(k)=k$.

To prove \Cref{prp:cyclotomic_many_squares}, we need some results concerning $A_w(\alpha)$.

\begin{lem}\label{lem:Aw_minimal_positive_value}
For $1\le w \le \frac{p^{n-1}-1}2$, $A_w(\alpha)$ is even. In particular, if $A_w(\alpha) \ne 0$, then $A_w(\alpha)\ge 2$.
\end{lem}
\begin{proof}
This follows immediately by reducing \eqref{eq:Aw-2} modulo 2, using that $R_w$ is a permutation on $I_w^\ast$.
\end{proof}
We can also specify the possible values of the $a_k$'s when $A_w(\alpha) = 2$.
\begin{lem}\label{lem:Aw2_cases}
Let $1 \le w \le \frac{p^{n-1}-1}2$, and suppose $A_w(\alpha) = 2$. 
Then there exists $k \in I_w$ and $s \in \lbrace 1 \ldots, p-1 \rbrace$ such that $a_{k} = a_{R_w(k)} = a_{R_w^2(k)} = \cdots = a_{R_w^{s-1}(k)} \in \{ \pm 1 \}$ and $a_{R_w^i(k)} = 0$ for $i \in \lbrace s, \ldots, p-1 \rbrace$.
\end{lem}
\begin{proof}
It follows by \eqref{eq:Aw-2} that there must be exactly $2$ elements $k \in I_w^\ast$ for which $a_k \neq a_{R_w(k)}$, and in this case, either $a_k = 0$ and $a_{R_w(k)} \in \{ \pm 1 \}$, or vice versa.
As $R_w$ is a cyclic permutation on $I_w^\ast$, the statement follows.
\end{proof}

\begin{lem}\label{lem:cross_term_nonvanishing}
Let $U\subseteq\{0,1,\ldots,n-2\}$ with $|U|\ge 2$, and $\alpha \in \mathcal O_{K_{p^n}^+}$ be such that
\begin{itemize}
    \item when $\alpha$ is written as in \eqref{eq:cyclotomic_expansion}, we have $a_k = 0$ unless $k \equiv \pm p^u\pmod{p^{n-1}}$ for some $u\in U$; and
    \item $A_{p^u}(\alpha)= 2$ for all $u\in U$, and $A_{p^u}(\alpha)= 0$ for $u\not\in U$.
\end{itemize}
Let $V = \{v_1,v_2\}\subseteq\{0,1,\ldots,n-2\}$.
\begin{enumerate}
    \item If $V\not\subseteq U$, then $C_m(\alpha^2) = 0$ for every integer $m$ with $0 \leq m \leq \frac{p-1}{2}p^{n-1} - 1$ and $m\equiv \pm(p^{v_1}-p^{v_2}) \pmod{p^{n-1}}$.
    \item If $V\subseteq U$, then $C_m(\alpha^2)\ne 0$ for some integer $m$ with $0 \leq m \leq \frac{p-1}{2}p^{n-1} - 1$ and $m\equiv \pm(p^{v_1}- p^{v_2}) \pmod{p^{n-1}}$.
\end{enumerate}
\end{lem}
In the above statement and in its proof, when we write $x \equiv \pm y \pmod{z}$ for non-zero integers $x, y$ and $z$ we mean that either $x \equiv y \pmod{z}$ or $x \equiv -y \pmod{z}$.
\begin{proof}
We start with the formula \eqref{eq:a2_expand} for $\alpha^2$. By collecting the coefficients of the symbols $\omega(k)$ in \eqref{eq:a2_expand} and identifying $\omega(p^n-k)$ with $\omega(k)$ using \eqref{eq:rf}, we write
\[
\alpha^2 = d_0 + \sum_{l=1}^{\frac{p^n-1}2} d_l \omega(l),
\]
where $d_l$ is the sum of coefficients of the symbols $\omega(l)$ and $\omega(p^n-l)$ in \eqref{eq:a2_expand} for $1\le l \le \frac{p^n-1}2$, and $d_0$ denotes the constant term in \eqref{eq:a2_expand}. Since we have $a_k=0$ unless $k\equiv \pm p^u \pmod{p^{n-1}}$ for some $u\in U$, it follows that $d_l$ can be nonzero only when $l\equiv 0, \pm p^{u_1}, \pm p^{u_1}\pm p^{u_2}\pmod{p^{n-1}}$ for some (not necessarily distinct) $u_1,u_2 \in U$. Meanwhile, using the conversion formula \eqref{eq:d2}, we find for $0\le s \le \frac{p-3}2$ and $1\le w \le p^{n-1}-1$ that
\begin{equation}\label{eq:two_terms}
C_{sp^{n-1}+w}(\alpha^2) = \begin{cases} d_{sp^{n-1}+w} - d_{\frac{p-1}2p^{n-1}+w} & \text{if } 1\le w \le \frac{p^{n-1}-1}2\\ d_{sp^{n-1}+w} - d_{\frac{p+1}2p^{n-1}-w} & \text{if } \frac{p^{n-1}+1}2 \le w \le p^{n-1}-1.\end{cases}
\end{equation}
Crucially, given that $\frac{p-1}{2}p^{n-1}+w=\frac{p+1}{2}p^{n-1}- (p^{n-1}-w)$, we have that the term $d_{\frac{p-1}{2}p^{n-1}+w}$ appears in both the top and the bottom equalities above, for $1\le w \le \frac{p^{n-1}-1}2$.

Suppose $V = \{v_1,v_2\} \not\subseteq U$, and let $m$ be an integer with $1 \leq m \leq \frac{p-1}{2}p^{n-1} - 1$ and $m \equiv \pm(p^{v_1} - p^{v_2}) \pmod{p^{n-1}}$. By the reasoning above, we find that $d_l = 0$ for all integers $l$ with $0 \leq l \leq \frac{p^n-1}{2}$, $l\equiv\pm p^{v_1}\pm p^{v_2}$ (noting that $p\ge 5$). It then follows from \eqref{eq:two_terms} that $C_m(\alpha^2) = 0$. This gives (1).

Now suppose $V = \{v_1,v_2\} \subseteq U$.
For an integer $l$ with $1 \leq l \leq \frac{p^n - 1}{2}$ and $l \equiv \pm(p^{v_1} - p^{v_2})\pmod{p^{n-1}}$, the coefficient $d_l$ is given by
\begin{equation}\label{eq:d_l}
    d_l = \sum_{\substack{1\le k_1, k_2 \le \frac{p-1}2p^{n-1}-1\\ (k_1,k_2) \equiv (\pm p^{v_1}, \pm p^{v_2})\pmod{p^{n-1}}\\ k_1 \pm k_2 \equiv \pm l \pmod{p^n}}} 2a_{k_1}a_{k_2}. 
\end{equation}
The terms appearing in the sum above can be described in terms of $R_w$ introduced in \eqref{eq:R-defi}; we will simply write $R$ instead of $R_w$ to ease the notation. Specifically, consider some fixed integers $k_1, k_2$ with $1 \leq k_1, k_2 \leq \frac{p-1}{2}p^{n-1} - 1$ where
\[
k_1 \equiv \pm p^{v_1}\pmod{p^{n-1}}, \quad k_2 \equiv \pm p^{v_2}\pmod{p^{n-1}},
\]
and such that there exists $1\le l \leq \frac{p^n - 1}{2}$ with $l\equiv \pm(p^{v_1}-p^{v_2})\pmod{p^{n-1}}$ and 
\[k_1 + \epsilon k_2\equiv \pm l \pmod{p^n},\] 
for some $\epsilon \in \{\pm 1\}$. Then we also have $R(k_1) + \epsilon R(k_2)\equiv \pm l \pmod{p^n}$.
For our choice of $l$, $k_1$ and $k_2$, the equation \eqref{eq:d_l} can be rewritten as 
\begin{equation}\label{eq:d_l-2}
    d_l=2\sum^{p-1}_{i=0} a_{R^i(k_1)}a_{R^i(k_2)}.
\end{equation}
In other words, the pattern by which the indices of the coefficients iterate matches paths in the diagrams associated with $A_{p^{v_1}}(\alpha)$ and $A_{p^{v_2}}(\alpha)$ (as in \eqref{eq:diagram}), starting form ${k_1}$ and ${k_2}$, respectively.  

From \eqref{eq:two_terms}, we see that if $C_m(\alpha^2) = 0$ for every integer $m$ with $1 \leq m \leq \frac{p-1}{2}p^{n-1} - 1$ with $m\equiv \pm (p^{v_1}-p^{v_2})\pmod{p^{n-1}}$, then the coefficients $d_l$ for integers $l$ with $1 \leq l \leq \frac{p^{n} - 1}{2}$ with $l\equiv \pm (p^{v_1}-p^{v_2})\pmod{p^{n-1}}$ must all be equal. Keeping this in mind, we evaluate the coefficients $d_l$. By \Cref{lem:Aw2_cases}, the assumption that $A_{p^{v_1}}(\alpha) = A_{p^{v_2}}(\alpha) = 2$ implies that 
that there are $k_1 \in I_{p^{v_1}}$ and $k_2 \in I_{p^{v_2}}$ and $s_1, s_2 \in \lbrace 1, \ldots p - 1 \rbrace$ such that $k_1+\epsilon k_2\equiv \pm(p^{v_1}-p^{v_2})\pmod{p^{n-1}}$, and (for $j \in \lbrace 1, 2 \rbrace$) that $a_{k_j} = a_{R^i(k_j)} \in \{ \pm 1\}$ for $i \in \lbrace 0, 1, \ldots, s_j-1 \rbrace$ and $a_{R^i(k_j)} = 0$ for $j \in \lbrace 1, 2\rbrace$ and $i \in \lbrace s_j, \ldots, p-1 \rbrace$.
Consider now positive integers $1\le l_1,l_2,l_3 \le \frac{p^n-1}2$ such that $l_1\equiv \pm (k_1+\epsilon k_2) \pmod{p^n}$, $l_2\equiv \pm(R^i(k_1)-\epsilon R^{i-1}(k_2))\pmod{p^n}$, and $l_3\equiv \pm (R^i(k_1)-\epsilon R^{i+1}(k_2))\pmod{p^n}$ for some $i$ for which $R^{i}(k_j)$, $R^{i+1}(k_2)$ and $R^{i-1}(k_2)$ are not $\ast$ for $j \in \lbrace 1, 2 \rbrace$; it is easy to see that the quantities $l_2,l_3$ do not depend on the choice of $i$. From equation \eqref{eq:d_l-2} we see that $d_{l_1},d_{l_2}$ and $d_{l_3}$ are equal to $\pm 2\min\{s_1,s_2\}$, $\pm 2\min\{s_1-1,s_2\}$ and $\pm 2\min\{s_1,s_2-1\}$, which cannot all be equal. Therefore we get a contradiction, and $C_m(\alpha^2)$ cannot be zero for all $m\equiv \pm (p^{v_1}- p^{v_2})\pmod{p^{n-1}}$. This proves (2).
\end{proof}

\begin{proof}[Proof of \Cref{prp:cyclotomic_many_squares}]
It suffices to prove that $t_m$ cannot be expressed as a sum of $m-2$ squares in $\mc O_{K_{p^n}^+}$ for any $m\le n\in\N$. For fixed $n\ge m$, recall that we defined $\omega(k) = \omega(p^n,k) = \zeta_{p^n}^k + \zeta_{p^n}^{-k}$, and we thus may rewrite
\[
t_m = \sum_{k=n-m}^{n-2} \omega(p^k)^2 =\sum_{k=n-m}^{n-2} (\omega(2p^k) + 2).
\]
Let
\[
t_m = \alpha_1^2 + \cdots + \alpha_l^2, \quad \alpha_j = a_{j,0} + \sum_{k=1}^{\frac{p-1}2 p^{n-1}-1} a_{j,k} \omega(k)
\]
be a representation of $t_m$ as a sum of $l$ squares in $\mc O_{K_{p^n}^+}$. Taking the constant coefficient, we obtain
\begin{equation}\label{eq:cyclotomic_sum_constant}
C_0(t_m) = 2(m-1) = \sum_{j=1}^l \sum_{i=0}^{\frac{p^{n-1}-1}{2}} A_i(\alpha_j).
\end{equation}
Meanwhile, for $n-m \le k \le n-2$, taking the $2p^k$-th coefficient gives
\begin{equation}\label{eq:cyclotomic_sum_2pk}
C_{2p^k}(t_m) = 1.
\end{equation}

Since each $A_i(\alpha_j)$ (for $0 \leq i \leq \frac{p^{n-1}-1}{2}$ and $1 \leq j \leq l$) is either $0$ or at least $2$ by \Cref{lem:Aw_minimal_positive_value}, \eqref{eq:cyclotomic_sum_constant} implies that there can be at most $m-1$ pairs $(i, j)$ for which $A_{i}(\alpha_j) \neq 0$.
We will show that for each $n - m \leq k \leq n - 2$ there is some index $j$ for which $A_{p^k}(\alpha_j) \neq 0$; this will then imply that this $j = j(k)$ is unique, that $A_i(\alpha_{j(k)}) = 2$, and that $A_i(\alpha_{j'}) = 0$ for all other values of $(i, j')$.
To this end, observe that it follows from \eqref{eq:cyclotomic_sum_2pk} that $C_{2p^k}(\alpha_j^2)$ is odd for some $j$. Using the congruence \eqref{eq:cyclotomic_mod2}, as well as \eqref{eq:rf} and \eqref{eq:d2}  
for the conversions, we see that $C_{2p^k}(\alpha_j^2)$ is odd if and only if exactly one of $a_{j,p^k}$ and $a_{j,\frac{p-1}4p^{n-1}+p^k}$ is odd. Observing that $A_{p^k}(\alpha_{j})=0$ if and only if $a_{j,l} = 0$ for every $l\equiv \pm p^k \pmod{p^{n-1}}$, this implies that $A_{p^k}(\alpha_{j})$ is nonzero. 
This establishes the claim.

Finally, we claim that for $k_1\ne k_2$, then $j(k_1)\ne j(k_2)$. Suppose to the contrary that $j(k_1)=j(k_2)=j_0$. Then, by \Cref{lem:cross_term_nonvanishing}$(2)$, we have $C_l(\alpha_{j_0}^2) \ne 0$ for some $l \equiv \pm(p^{k_1}-p^{k_2}) \pmod{p^{n-1}}$, and by \Cref{lem:cross_term_nonvanishing}$(1)$ it follows that $C_l(\alpha_j^2) = 0$ for $j\ne j_0$, noting that $A_{p^{k_1}}(\alpha_j)=A_{p^{k_2}}(\alpha_j)=0$. This implies $C_l(t_m) \ne 0$, a contradiction. This establishes the claim. From the claim we conclude that the indexing function
\[
j:\cb{n-m,n-m+1,\ldots,n-2} \to \{1,\ldots,l\}
\]
is injective. This says $l \ge m-1$, and therefore $t_m$ cannot be written as a sum of $m-2$ squares.
\end{proof}

\begin{rmk}
With a similar argument, it can be shown that $P(\mc O_{K_{p^\infty}^+}) = \infty$ also for $p=2$ and $p\equiv 3\pmod{4}$.
\end{rmk}

\section{Open problems}\label{sect:Open}
Let $K$ be an infinite real multiquadratic extension of $\qq$. We conjecture that $P(\mc O_K)=\infty$. It would be tantalising to predict the same behaviour for any infinite totally real algebraic field $K$ with the Northcott property, say, to draw a parallel with the non-existence of universal quadratic forms. Just like for non-existence of universal quadratic forms \cite[Proposition 4.6]{DKMWY}, we also have an example, in $K^+_{p^{\infty}}$, of infinite totally real algebraic fields without the Northcott property but with an infinite Pythagoras number.  

We conclude by asking the following questions.

\begin{enumerate}
    \item How to determine if the Pythagoras number of the ring of integers of an infinite totally real algebraic extension is finite or infinite?
    \item Do all natural numbers appear as Pythagoras numbers for the rings of integers of such fields?
    \item Does there exist an infinite real cyclotomic field whose ring of integers has a finite Pythagoras number?
    \item Does there exist a lower bound for the Pythagoras number of the ring of integers of a totally real number field $K$ with respect to $[K:\Q]$?
    \item Do there exist infinitely many totally real number fields $K$ such that $P(\mc O_K)=3$?
\end{enumerate}
\section{Declarations}
\subsection*{Acknowledgments}
We thank Jakub Krásenský and Martin Widmer for providing helpful input and pointers. We also thank the number theory group at Charles University for useful discussions on this topic and the referees for their thorough reading of the manuscript and numerous constructive suggestions.

The first author would like to thank the Hausdorff Research Institute for Mathematics, Bonn, funded by the DFG (under Germany's Excellence Strategy, EXC-2047/1 – 390685813), for its hospitality during the trimester program ``Definability, decidability, and computability'' while this paper was revised.

\subsection*{Conflict of Interest/Competing Interest}
	None.
\subsection*{Funding statement}
The authors were supported by \emph{Charles University} PRIMUS Research Programmes PRIMUS/24/SCI/010 (Daans, Yatsyna) and PRIMUS/25/SCI/008 (Man), \emph{Charles University} programme UNCE/24/SCI/022 (Yatsyna), \emph{Czech Science Foundation} GAČR, grant 21-00420M (Gajović, Man), Junior Fund grant for postdoctoral positions at \emph{Charles University} (Gajović), MPIM guest postdoctoral fellowship programme (Gajović), and \emph{Research Foundation--Flanders (FWO)}, fellowship 1208225N (Daans).

\bibliographystyle{alpha}
\bibliography{bibliography}

\end{document}